\documentclass[12pt]{amsart}
\usepackage[english]{babel}
\usepackage{amsfonts}
\usepackage{amsthm}
\usepackage{amsbsy}
\usepackage{amssymb}
\usepackage{graphicx}
\usepackage{eso-pic}
\usepackage{tikz}
\usepackage{xfrac}
\usepackage{float}
\usepackage[bookmarks=true,colorlinks=true,urlcolor=black,linkcolor=black,citecolor=black,pagebackref=false]{hyperref}
\usepackage{resizegather}
\usepackage{bm}
\usepackage{bookmark}
\usepackage{amsmath}
\usepackage{subfigure}
\usepackage{hyperref}
\usepackage{scalerel,stackengine}
\usepackage{epigraph}
\usetikzlibrary{matrix,arrows}
\usepackage[letterpaper,margin=1.0in,top=1.25in,bottom=1.3in]{geometry}
\usepackage{microtype}
\makeatletter
\newcommand\footnoteref[1]{\protected@xdef\@thefnmark{\ref{#1}}\@footnotemark}
\makeatother

\newtheorem{lemma}{Lemma}[section]
\newtheorem{thm}[lemma]{Theorem}
\newtheorem{prop}[lemma]{Proposition}
\newtheorem{cor}[lemma]{Corollary}
\newtheorem*{cor*}{Corollary}

\theoremstyle{definition}
\newtheorem{defn}[lemma]{Definition}

\theoremstyle{remark}

\stackMath
\newcommand\reallywidehat[1]{%
\savestack{\tmpbox}{\stretchto{%
  \scaleto{%
    \scalerel*[\widthof{\ensuremath{#1}}]{\kern-.6pt\bigwedge\kern-.6pt}%
    {\rule[-\textheight/2]{1ex}{\textheight}}
  }{\textheight}%
}{0.5ex}}%
\stackon[1pt]{#1}{\tmpbox}%
}
\author{Leonardo Ferrari}
\address{Dipartimento di Matematica, Universit\`{a} di Pisa, Largo Bruno Pontecorvo 5, Pisa, Italy}
\email{leonardocpferrari at gmail dot com}

\author{Alexander Kolpakov}
\address{Institut de Math\'ematiques, Universit\'e de Neuch\^atel, Rue Emile-Argand 11, Neuch\^atel, Suisse / Switzerland}
\email{kolpakov dot alexander at gmail dot com}

\author{Leone Slavich}
\address{Dipartimento di Matematica ``F. Casorati'', Universit\`{a} di Pavia, Via Ferrata 5, Pavia, Italy}
\email{leone dot slavich at gmail dot com}

\thanks{A.K. was supported by the Swiss National Science Foundation project no. PP00P2-170560.}

\title[Cusps of hyperbolic $4$-manifolds and homology spheres]{Cusps of hyperbolic $4$-manifolds \\ and rational homology spheres}

\begin{document}

\begin{abstract}
In the present paper, we construct a cusped hyperbolic $4$-manifold with all cusp sections homeomorphic to the Hantzsche--Wendt manifold, which is a rational homology sphere. By a result of Gol\'enia and Moroianu, the Laplacian on $2$-forms on such a manifold has purely discrete spectrum. This shows that one of the main results of Mazzeo and Phillips from 1990 cannot hold without additional assumptions on the homology of the cusps. This also answers a question by Gol\'enia and Moroianu from 2012. 

We also correct and refine the incomplete classification of compact orientable flat $3$-manifolds arising from cube colourings provided earlier by the last two authors. 
\end{abstract}

\maketitle

\section{Introduction}

A Riemannian $n$--manifold $\mathcal{M}$ with a complete metric of constant sectional curvature $-1$ is called a {\itshape hyperbolic $n$--manifold}. Equivalently, one may think of $\mathcal{M}$ as a quotient space of the $n$--dimensional hyperbolic space $\mathbb{H}^n$ by a discrete torsion-free subgroup $\Gamma < \mathrm{Isom}(\mathbb{H}^n)$ of isometries. By replacing $\mathbb{H}^n$ with $\mathbb{X}^n = \mathbb{S}^n$ (spherical space) or $\mathbb{E}^n$ (Euclidean space) we can analogously define spherical (with constant sectional curvature $+1$) and Euclidean (or flat, with constant sectional curvature $0$) $n$--manifolds. 

Hyperbolic $n$--manifolds of finite volume split into two natural classes: compact and cusped ones. The latter ones are necessarily non-compact, and Margulis' lemma \cite[Theorem D.1.1]{BP} implies that their topological ends have the form $\mathcal{E} \times [0, +\infty)$ with warped product metric, where $\mathcal{E}$ is a compact Euclidean $(n-1)$--manifold called a cusp section. 

In this paper we are interested in understanding the possible structure of cusps of finite-volume hyperbolic $4$--manifolds and, in particular, we give a positive answer to a question posed by Gol\'enia and Moroianu \cite[Question 1]{GM}.

\begin{thm}\label{teo:main}
There exists a finite-volume orientable hyperbolic $4$--manifold $\mathcal{M}$ with all cusp sections homeomorphic to the Hantzsche--Wendt manifold.
\end{thm}

The Hantzsche--Wendt manifold is the only rational homology sphere among the $6$ possible homeomorphism types of compact orientable Euclidean $3$-manifolds (see \cite{HW} or \cite[Corollary 3.5.10]{Wolf}). To the best of our knowledge, the manifold $\mathcal{M}$ of Theorem \ref{teo:main} is the first example of an orientable complete non-compact finite volume hyperbolic $n$-manifold all of whose cusp sections are rational homology $(n-1)$-spheres.  

The other interesting fact is that the spectrum of the Laplacian $\Delta_2$ acting on $2$-forms on $\mathcal{M}$ is purely discrete by \cite[Theorem 1.2]{GM}, and this is the first example of an orientable hyperbolic manifold with such property. This contradicts the work of Mazzeo and Phillips \cite[Theorem 1.11]{MP}, which states that all finite volume, cusped hyperbolic $n$-manifolds should have non-empty continuous spectrum for the Laplacian $\Delta_k$ on $k$-forms, for all $k=0,\dots,n$. 

The issue here is that Mazzeo and Phillips' result \cite[Theorem~1.11]{MP} apparently holds under the tacit assumption that the cusped hyperbolic manifold $\mathcal{M}$ is orientable and that no cusp section of $\mathcal{M}$ has two consecutive Betti numbers equal to zero. Moreover, Gol\'enia and Moroianu \cite{GM} conjecture that the latter condition on non-vanishing of the Betti numbers cannot be taken for granted, which Theorem~\ref{teo:main} shows indeed to be the case. The fact that spectra of differential operators on cusped hyperbolic manifolds can be purely discrete was first remarked by B\"ar in \cite[Theorem~1]{Baer} for Dirac operators.

Also we note that the classification from \cite[Proposition 3.2]{KS?} of the so-called proper colourings of the $3$--dimensional cube is incomplete, and use the opportunity to correct the omissions. The resulting new classification turns out to be more succinct and transparent. 

\subsection*{Outline of the paper}
In Section \ref{sec:preliminaries} we review colourings on right-angled polytopes, how they can be used to define hyperbolic manifolds, and some of their properties. In Section \ref{sec:cubecolorings} we classify up to isometry all orientable manifolds arising as colourings of a $3$-dimensional unit cube, correcting and refining the incomplete classification from \cite[Proposition 3.2]{KS?}. In Section \ref{sec:24-cell-coloring} we introduce the right-angled hyperbolic $24$-cell, which is an ideal hyperbolic $4$-polytope, and exhibit a colouring which produces the manifold $\mathcal{M}$ of Theorem \ref{teo:main}. 

\section{Preliminaries}\label{sec:preliminaries}

A finite-volume polytope $\mathcal{P} \subset \mathbb{X}^n$ (for $\mathbb{X}^n = \mathbb{S}^n, \mathbb{E}^n, \mathbb{H}^n$ being spherical, Euclidean and hyperbolic $n$--dimensional space, respectively, cf. \cite[Chapters 1--3]{Ratcliffe}) is called {\itshape right-angled} if any two codimension $1$ faces (or facets, for short) are either intersecting at right angles or disjoint. It is known that hyperbolic right-angled polytopes cannot exist if $n > 12$ \cite{Dufour, PV}\footnote{The upper bound for the dimension of a right-angled finite-volume hyperbolic polytope was shown to be 14 in \cite{PV}, and then improved to 12 in \cite{Dufour}.}, and examples are known up to dimension $n = 8$ \cite{PV}. A polytope is called {\itshape ideal} if all its vertices are located on the ideal boundary $\partial \mathbb{H}^n$. Ideal right-angled polytopes cannot exist if $n>8$, and the only examples known have dimensions $n \leq 4$ \cite{Kolpakov}.

\subsection{Colourings of right-angled polytopes}
One of the important properties of hyperbolic right-angled polytopes is that their so-called colourings provide a rich class of hyperbolic manifolds. By inspecting the combinatorics of a colouring, one may obtain important topological and geometric information about the associated manifold.

\begin{defn}
Let $\mathcal{S}$ be an $n$--dimensional simplex with the set of vertices $\mathcal{V}=\{v_1,\dots,v_{n+1}\}$ and $W$ an $\mathbb{F}_2$-vector space. A map $\lambda: \mathcal{V} \rightarrow W$ is called a {\itshape colouring} of $\mathcal{S}$. The colouring of $\mathcal{S}$ is {\itshape proper} if the vectors $\lambda(v_i), \ i=1,\dots, n+1$ are linearly independent. 

Let $K$ be a simplicial complex with the set of vertices $\mathcal{V}$. Then a {\itshape colouring} of $K$ is a map $\lambda: \mathcal{V} \rightarrow W$. A colouring $\lambda$ of $K$ is {\itshape proper} if $\lambda$ is proper on each simplex in $K$. 
\end{defn}

Notice that if $\mathcal{P} \subset \mathbb{X}^n$ is a compact right-angled polytope then $\mathcal{P}$ is necessarily simple and the boundary complex $K_\mathcal{P}$ of its dual $\mathcal{P}^*$ is a simplicial complex.

If the polytope $\mathcal{P} \subset \mathbb{H}^n$ is right-angled and non-compact, it will necessarily have some ideal vertices, and the boundary of its dual $\mathcal{P}^*$ will not be simplicial anymore. However, all of its codimension $k>1$ faces will be simplices, and we can consider the maximal simplicial subcomplex $K_\mathcal{P}$ of the boundary of $\mathcal{P}^*$.

\begin{defn}
Let $\mathcal{P} \subset \mathbb{X}^n$ be a polytope with the set of facets $\mathcal{F}$ and $K_\mathcal{P}$ be the maximal simplicial subcomplex of the boundary of $\mathcal{P}^*$. A \textit{colouring} of $\mathcal{P}$ is a map $\lambda: \mathcal{F} \rightarrow W$, where $W$ an $\mathbb{F}_2$-vector space. This map naturally defines a colouring of $K_\mathcal{P}$. Then $\lambda$ is called \textit{proper} if the induced colouring on $K_\mathcal{P}$ is proper. 
\end{defn}

If the polytope $\mathcal{P}$ or the vector space $W$ are clear from the context, then we will omit them and simply refer to $\lambda$ as a colouring. The \textit{rank} of $\lambda$ is the $\mathbb{F}_2$-dimension of $\mathrm{im}\, \lambda$. We will always assume that colourings are surjective, in the sense that the image of the map $\lambda$ is a generating set of vectors for $W$.

A colouring of a right-angled $n$-polytope $\mathcal{P}$ naturally defines a homomorphism, which we still denote by $\lambda$ without much ambiguity, from the associated right-angled Coxeter group~$\Gamma(\mathcal{P})$ (generated by reflections in all the facets of $\mathcal{P}$) into $W$ (with its natural group structure). Being a Coxeter polytope, $\mathcal{P}$ has a natural orbifold structure as the quotient $\mathbb{X}^n /_{\Gamma(\mathcal{P)}}$.

\begin{prop}[\cite{KS?}, Proposition 2.1]
If the colouring $\lambda$ is proper, then $\ker \lambda < \Gamma(\mathcal{P})$ is torsion-free, and $\mathcal{M}_\lambda = \mathbb{X}^n /_{\ker \lambda}$ is a manifold.
\end{prop}

If the dimension of the vector space $W$ is minimal, i.e. equal to the maximum number of vertices in the simplices of $\mathcal{P}^*$, $\mathcal{M}_\lambda$ is called a \textit{small cover} of $\mathcal{P}$. This is equivalent to $\text{dim} \, _{\mathbb{F}_2} W =n$ if $\mathcal{P}$ has proper hyperbolic vertices, and $\text{dim} \, _{\mathbb{F}_2} W =n-1$ if $\mathcal{P}$ is ideal.

\begin{defn}
Let $\mathrm{Sym}(\mathcal{P})$ be the group of symmetries of a polytope $\mathcal{P}\subset \mathbb{X}^n$ having set of facets $\mathcal{F}$. Two $W$--colourings $\lambda,\mu:\mathcal{F}\to W$ of $\mathcal{P}$ are called \textit{equivalent} (or \textit{$DJ$--equivalent}, after Davis and Januszkiewicz \cite{DJ}) if there exists a combinatorial symmetry $s \in \mathrm{Sym}(\mathcal{P})$ and an invertible linear transformation $m \in \mathrm{GL}(W)$ such that $\lambda = m \circ \mu \circ s$. 
\end{defn}

It is easy to see that $DJ$-equivalent proper colourings of a polytope $\mathcal{P} \subset \mathbb{X}^n$ define isometric manifolds.

\subsection{Ideal vertices and cusp sections}\label{sec:ideal-vertices} When $\mathcal{P}$ has ideal vertices, it is also natural to determine the Euclidean cusp sections of the manifold $\mathcal{M}_{\lambda}$ (see \cite{KS?}). If $p$ is an ideal vertex of $\mathcal{P}$, its Euclidean vertex figure $E_p$ is the intersection of $\mathcal{P}$ with a sufficiently ``small'' horosphere centred at $p$. Each facet of $E_p$ will correspond to a unique facet of $\mathcal{P}$ containing~$p$.

Since $\mathcal{P}$ is right-angled, the vertex figure $E_p$ has to be a right-angled Euclidean polytope of dimension $n-1$, uniquely determined up to homothety. This polytope will correspond to a fundamental domain for the maximal affine Coxeter subgroup $\Gamma_p < \Gamma(\mathcal{P})$ generated by the reflections in the facets of $\mathcal{P}$ containing $p$.  If $\mathcal{P}$ is regular then also $E_p$ has to be regular, and is therefore forced to be an $(n-1)$-cube. Given the colouring $\lambda$, let us consider its restriction to the subgroup $\Gamma_p$. The kernel of this restriction will correspond to a maximal parabolic subgroup in $\ker \lambda$. We thus see that the vertex $p$ lifts to a number of cusps in the manifold $\mathcal{M}_{\lambda}$, and for each such cusp the Euclidean cusp section corresponds to the manifold obtained by restricting the colouring of $\mathcal{P}$ to $E_p$. The number of copies of each cusp is equal to $[W:W']$, where $W'$ is the subspace of $W$ generated by the colours of the facets of $E_p$. 

\subsection{Symmetries of colourings}
Given a (not necessarily proper) colouring $\lambda$ of a right-angled polytope $\mathcal{P}$, there is a natural group of symmetries of the associated orbifold $\mathcal{M}_{\lambda}$ called its \textit{coloured isometry group}. We briefly recall its definition from \cite[Section 2.1]{KS?}.

\begin{defn}\label{def:admissible-symmetries}
Let $\lambda$ be a $W$--colouring of $\mathcal{P}$. A symmetry $s$ of $\mathcal{P}$ is \textit{admissible} with respect to $\lambda$ if:
\begin{enumerate}
    \item the maps $s$ induces a permutation of the colours assigned by $\lambda$ to the facets of $\mathcal{P}$,
    \item such permutation is realised by an invertible linear automorphism $\phi \in \mathrm{GL}(W)$. 
    \end{enumerate}
\end{defn}

Admissible symmetries are easily seen to form a subgroup, which we denote by $\mathrm{Adm}_{\lambda}(\mathcal{P})$, of the symmetry group of $\mathcal{P}$ and there is a naturally defined homomorphism $\Phi$ from $\mathrm{Adm}_{\lambda}(\mathcal{P})$ to $\mathrm{GL}(W)$.

Recall that a colouring $\lambda:\Gamma(\mathcal{P})\rightarrow W$  defines a regular orbifold cover $\pi: \mathcal{M}_{\lambda}\rightarrow \mathcal{P}$ with automorphism group $W$. The coloured isometry group $\mathrm{Isom}_c(\mathcal{M}_{\lambda})$ is defined as the group of symmetries of $\mathcal{M}_{\lambda}$ which are lifts of admissible symmetries of $\mathcal{P}$. There is a short exact sequence
\begin{equation*}
    1\rightarrow W \rightarrow \mathrm{Isom}_c(\mathcal{M}_{\lambda})\rightarrow \mathrm{Adm}_{\lambda}(\mathcal{P}) \rightarrow 1
\end{equation*}
which clearly splits, so that 
\begin{equation}\label{eq:semidirect_product}
    \mathrm{Isom}_c(\mathcal{M}_{\lambda})\cong W \rtimes\mathrm{Adm}_{\lambda}(\mathcal{P}).
\end{equation}
The action of $\mathrm{Adm}_{\lambda}(P)$ on $W$ is precisely the one induced by the homomorphism $\Phi$.

 One advantage of colourings with large groups of admissible symmetries is that properties such as properness, or the shape of eventual cusps, can be determined by looking at fewer conditions and ``pushing forward'' through the action of $\mathrm{Adm}_{\lambda}(\mathcal{P})$.

\subsection{Computing the Betti numbers of colourings.} Given a right-angled polytope $\mathcal{P} \subset \mathbb{X}^n$ with an $\mathbb{F}^k_2$--colouring $\lambda$, let us enumerate the facets $\mathcal{F}$ of $\mathcal{P}$ in some order. Then, we may think that $\mathcal{F} = \{ 1, 2, \ldots, m \}$. Let $\Lambda$ be the \textit{defining matrix} of $\lambda$, that consists of the column vectors $\lambda(1), \ldots, \lambda(m)$ exactly in this order. Then $\Lambda$ is a matrix with $k$ rows and $m$ columns. 

The manifold $\mathcal{M}_{\lambda}$ is homotopy equivalent to a non-positively curved cube complex $\mathcal{C}_{\lambda}$ obtained by dualising the tessellation of $\mathcal{M}_{\lambda}$ into copies of the polytope $\mathcal{P}$, so that for every codimension-$k$ stratum in the tessellation of $\mathcal{M}_{\lambda}$, there is a $k$-dimensional cube in the complex $\mathcal{C}_{\lambda}$.

The cube complex $\mathcal{C}_{\lambda}$ admits a locally standard $\mathbb{F}_2^k$-action, induced by the automorphism group of the cover $\mathcal{M}_{\lambda}\rightarrow \mathcal{P}$. This action endows $\mathcal{C}_{\lambda}$ with the structure of a real toric space $M^{\mathbb{R}}(K,\Lambda)$, where $K=K_{\mathcal{P}}$ is the maximal simplicial subcomplex of the dual polytope $\mathcal{P}^*$ and $\Lambda:\mathbb{F}_2^m \rightarrow \mathbb{F}_2^k$ is the linear map represented by the defining matrix $\Lambda$ (see \cite{CP2}). 

Let $\mathrm{Row}(\Lambda)$ denote the row space of $\Lambda$, while for a vector $\omega \in \mathrm{Row}(\Lambda)$ let $K_\omega$ be the simplicial subcomplex of the complex $K$ spanned by the vertices $i$ (also labelled by the elements of $\{1, 2, \ldots, m\}$) such that the $i$--th entry of $\omega$ equals $1$.

If $R$ is a commutative ring in which $2$ is a unit, the co-homology of $M=M^{\mathbb{R}}(K,\Lambda)$ can be computed through the following formula \cite[Theorem 1.1]{CP2}:
\begin{equation}\label{eq:cohomology}
H^p(M,R)\cong \underset{\omega \in \mathrm{Row}(\Lambda)}{\bigoplus}\widetilde{H}^{p-1}(K_{\omega},R). 
\end{equation}

In what follows we will be interested primarily in computing the Betti numbers of $\mathcal{M}_\lambda$. By applying equation (\ref{eq:cohomology}) for $R=\mathbb{Q}$ we obtain the following corollary.

\begin{cor}\label{cor:CP}
Let $\mathcal{P} \subset \mathbb{X}^n$ be a right-angled polytope with a proper colouring $\lambda$, and let $\Lambda$ be the defining matrix of the latter.  Then for the manifold $\mathcal{M}_\lambda$ we have
\begin{equation*}
\beta^i(\mathcal{M}_\lambda) = \sum_{\omega \in \mathrm{Row}(\Lambda)} \widetilde{\beta}^{i-1}(K_\omega),
\end{equation*}
where $\beta^i$ denotes the $i$--th Betti number, and $\widetilde{\beta}^i$ is the $i$--th reduced Betti number. 
\end{cor}

We conclude this section by proving the following useful proposition.

\begin{prop}\label{prop:orientability_criterion}
Let $\mathcal{P} \subset \mathbb{X}^n$ be a right-angled polytope with a proper colouring $\lambda$. The manifold $\mathcal{M}_{\lambda}$ is orientable if and only if the vector $\varepsilon=(1,\dots,1)$ belongs to $\mathrm{Row}(\Lambda)$. 
\end{prop}
\begin{proof}
The manifold $\mathcal{M}_{\lambda}$ is orientable if and only if the kernel of the associated homomorphism $\lambda:\Gamma(\mathcal{P})\rightarrow \mathbb{F}_2^k$ contains only orientation preserving elements, i.e.\ elements which are expressed as the product of an even number of reflections in the facets of $\mathcal{P}$. With each element $\gamma$ of $\Gamma(\mathcal{P})$ we can associate a vector $v_{\gamma}$ in $\mathbb{F}_2^m$ as follows: if the generator corresponding to the $i$-th facet of $P$ appears an odd number of times in a word representing $\gamma$, then the $i$-th entry of $v_{\gamma}$ is equal to $1$, and is zero otherwise. Notice that $v_{\gamma}$ does not depend on the particular choice of a word representing $\gamma$, since all the relations in the presentation of $\Gamma(\mathcal{P})$ involve an even number (possibly zero) of each generator. Indeed the above construction yields a group homomorphism from $\Gamma(\mathcal{P})$ to $\mathbb{F}_2^m$, and it easy to see that $\gamma \in \Gamma(\mathcal{P})$ is orientation preserving if and only if $v_{\gamma}$ is a vector with an even number of entries equal to $1$.

Now, since the $m$-th generator of $\Gamma(\mathcal{P})$ is mapped by $\lambda$ into the $m$-th column of the matrix $\Lambda$, for any $\gamma \in \Gamma(\mathcal{P})$, $\lambda(\gamma)=\Lambda \cdot v_{\gamma}$, and therefore $\mathcal{M}_{\lambda}$ is orientable if and only if the right kernel of $\Lambda$ contains only vectors with an even number of $1$ entries. Finally, vectors with an even number of $1$ entries are precisely those vectors which are orthogonal to $\varepsilon=(1,\dots,1)$. Since $\mathrm{Row}(\Lambda)= (\ker \Lambda)^\perp$, the proof is complete.
\end{proof}

\section{Colouring the 3-cube}\label{sec:cubecolorings}

In this section we take the opportunity to correct the incomplete classification of colouring of the unit $3$-dimensional cube that first appeared in \cite{KS?}. This classification will be essential later on in order to perform the main construction of the paper.  

Let $\mathcal{C} = [0,1]^3 \subset \mathbb{R}^3$ be the unit cube. We label its faces as shown in Figure~\ref{fig:cube}, so that the pairs of parallel faces have indices $\{1,2\}$, $\{3,4\}$ and $\{5,6\}$. 

Given a colouring $\lambda: \Gamma(\mathcal{C}) \rightarrow W =  \mathbb{F}^k_2$, $k \geq 3$, the rank of the matrix $\Lambda$ clearly coincides with the rank of the colouring $\lambda$.
It is easy to see how the matrix $\Lambda$ changes under~$DJ$-equivalence: if we pre-compose the colouring with a symmetry $s  \in \mathrm{Sym}(\mathcal{C})$, we simply permute the columns of $\Lambda$ according to the action of the symmetry on the faces. If we post-compose the colouring with a linear transformation $m \in \mathrm{GL}(\mathbb{F}^k_2)$, we multiply $\Lambda$ on the left by the $k \times k$ matrix that represents the map $m$.
Clearly, the kernel of $\Lambda$ is invariant under post-composition with linear transformation of $\mathbb{F}^k_2$ , and so is $\mathrm{Row}(\Lambda) = (\ker \Lambda)^\perp$. 

\tikzset{every picture/.style={line width=0.75pt}} 

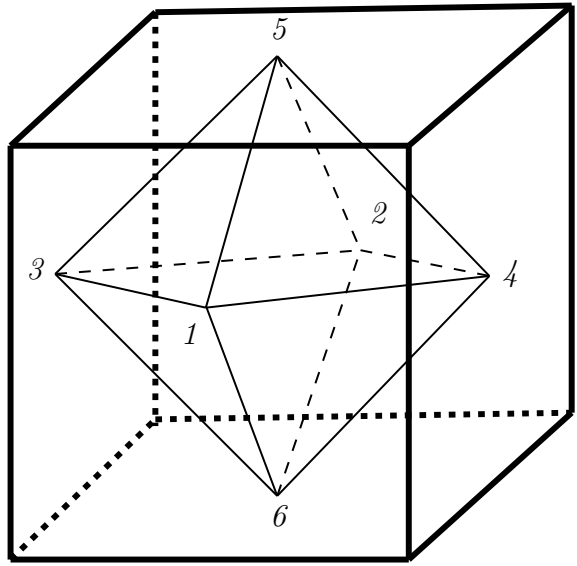
\begin{figure}
\centering
\begin{tikzpicture}[x=0.75pt,y=0.75pt,yscale=-1,xscale=1]

\draw    (427,30.5) -- (315,140.5) ;
\draw    (315,140.5) -- (339.87,146.06) -- (391,157.5) ;
\draw    (427,30.5) -- (391,157.5) ;

\draw    (315,140.5) -- (427,252.5) ;
\draw    (391,156.5) -- (427,252.5) ;

\draw    (391,157.5) -- (534,141.5) ;
\draw    (427,30.5) -- (534,141.5) ;
\draw    (534,141.5) -- (427,252.5) ;
\draw  [dash pattern={on 4.5pt off 4.5pt}]  (427,252.5) -- (469,128.5) ;
\draw  [dash pattern={on 4.5pt off 4.5pt}]  (427,30.5) -- (469,128.5) ;
\draw  [dash pattern={on 4.5pt off 4.5pt}]  (315,140.5) -- (469,128.5) ;
\draw  [dash pattern={on 4.5pt off 4.5pt}]  (469,128.5) -- (534,141.5) ;

\draw [line width=2.25]    (292.5,284.63) -- (493.34,284.63) ;
\draw [line width=2.25]  [dash pattern={on 2.53pt off 3.02pt}]  (292.5,284.63) -- (365.53,213.88) ;
\draw [line width=2.25]  [dash pattern={on 2.53pt off 3.02pt}]  (365.53,213.88) -- (575.5,210.51) ;
\draw [line width=2.25]    (493.34,284.63) -- (575.5,210.51) ;
\draw [line width=2.25]    (292.5,75.75) -- (493.34,75.75) ;
\draw [line width=2.25]    (292.5,75.75) -- (365.53,8.37) ;
\draw [line width=2.25]    (365.53,8.37) -- (575.5,5) ;
\draw [line width=2.25]    (493.34,75.75) -- (575.5,5) ;
\draw [line width=2.25]    (292.5,75.75) -- (292.5,284.63) ;
\draw [line width=2.25]    (493.34,75.75) -- (493.34,284.63) ;
\draw [line width=2.25]  [dash pattern={on 2.53pt off 3.02pt}]  (365.53,8.37) -- (365.53,217.25) ;
\draw [line width=2.25]    (575.5,5) -- (575.5,210.51) ;

\draw (377,164) node [anchor=north west][inner sep=0.75pt]   [align=left] {\textit{1}};
\draw (472,103) node [anchor=north west][inner sep=0.75pt]   [align=left] {\textit{2}};
\draw (422,256) node [anchor=north west][inner sep=0.75pt]   [align=left] {\textit{6}};
\draw (299,131) node [anchor=north west][inner sep=0.75pt]   [align=left] {\textit{3}};
\draw (422,10) node [anchor=north west][inner sep=0.75pt]   [align=left] {\textit{5}};
\draw (539,132) node [anchor=north west][inner sep=0.75pt]   [align=left] {\textit{4}};

\end{tikzpicture}

\caption{The cube $\mathcal{C}$ and its dual simplicial complex $\mathcal{O} = \mathcal{C}^*$ (the octahedron) with labelled faces (resp. vertices)}\label{fig:cube}
\end{figure}

Let $e_i$ be the standard $i$-th basis vector, i.e. a vector in $\mathbb{F}^k_2$ such that $(e_i)_j = \delta_{ij}$, for $1 \leq j \leq k$, where $\delta_{ij}$ is the Kronecker symbol. The \textit{$T$-set associated to a cube $\mathcal{C}$} as in Figure~\ref{fig:cube} is $T = \{ e_1 + e_2, \, e_3 + e_4, \, e_5 + e_6 \}$, which is understood as a subset of $\mathbb{F}^6_2$. In general, the $T$-set is a triple of vectors each having the form $e_i+e_j$, where the labels $i$ and $j$ correspond to a pair of opposite sides of $\mathcal{C}$: if the opposite sides of $\mathcal{C}$ have a different labelling, then the $T$-set has to be changed accordingly. 

\begin{prop}\label{prop:cube}
Let $\mathcal{C}$ be a cube and $T$ be its $T$-set. Let $\lambda: \Gamma(\mathcal{C}) \rightarrow \mathbb{F}^k_2$, $k\geq 3$, be a proper orientable colouring of $\mathcal{C}$ with defining matrix $\Lambda$, and let $\mathcal{M}_\lambda=\mathbb{R}^3 / \ker \lambda$ be the associated manifold. Then the following cases are possible:
\begin{enumerate}
    \item[(i)] if $\mathrm{Row}(\Lambda) \cap T = T$, then $\mathcal{M}_\lambda \cong \mathfrak{F}_1$ (the $3$-torus);
    \item[(ii)] if $\emptyset \neq \mathrm{Row}(\Lambda) \cap T \subsetneq T$, then $\mathcal{M}_\lambda \cong \mathfrak{F}_2$ (the ``half-twist'' manifold, i.e.\ the unique orientable Euclidean circle bundle over the Klein bottle);
    \item[(iii)] if $\mathrm{Row}(\Lambda) \cap T = \emptyset$, then $\mathcal{M}_\lambda \cong \mathfrak{F}_6$ (the Hantzsche-Wendt manifold).
\end{enumerate}
\end{prop}
\begin{proof}

The dual polytope of $\mathcal{C}$ is a simplicial complex isomorphic to the octahedron $\mathcal{O}$ in Figure~\ref{fig:cube}. Given the labelling on $\mathcal{C}$, the vertices of $\mathcal{O}$ become also labelled accordingly. Note that $\widetilde{\beta}^0(\mathcal{O}_\omega)$ equals $0$ if $\mathcal{O}_\omega$ is a connected, and $1$ otherwise. Moreover, $\mathcal{O}_\omega$ is disconnected only for $\omega = \{1, 2\}$, $\{3, 4\}$ and $\{5, 6\}$. By using Corollary~\ref{cor:CP}, we have that $\beta^1(\mathcal{M}_\lambda) = 3$ if and only if $T \subset \mathrm{Row}(\Lambda)$, in which case we have exactly three complexes $\mathcal{O}_\omega$ with $\widetilde{\beta}^0(\mathcal{O}_\omega) = 1$, for $\omega \in T$, and $\widetilde{\beta}^0(\mathcal{O}_\omega) = 0$, for $\omega \notin T$. It follows from \cite{HW} (see also \cite[Corollary 3.5.10]{Wolf}) that $\mathfrak{F}_1$ is the only closed flat orientable $3$-manifold with $\beta^1$ equal to $3$.

An analogous argument shows that $\beta^1(\mathcal{M}_\lambda) = 0$ if and only if $T \cap \mathrm{Row}(\Lambda) = \emptyset$. The only closed flat orientable $3$-manifold with vanishing $\beta^1$ is $\mathfrak{F}_6$ (see \cite{HW} or \cite[Corollary 3.5.10]{Wolf}).

The remaining case is $\emptyset \neq \mathrm{Row}(\Lambda) \cap T \subsetneq T$. Moreover, the intersection can contain only a single vector from $T$. Indeed, since $\lambda$ is an orientable colouring, by Proposition \ref{prop:orientability_criterion} the vector $\varepsilon = \sum^6_{i=1} e_i$ belongs to $\mathrm{Row}(\Lambda)$. This readily implies that whenever $\mathrm{Row}(\Lambda)$ contains two elements of $T$, it contains their sum with $\varepsilon$, which is the third element of $T$. 

Thus, we have that $\beta^1(M_\lambda) = 1$. This does not completely identify $\mathcal{M}_\lambda$: four out of six closed orientable flat $3$--manifolds satisfy this condition. However, all manifolds obtained from colourings have vanishing $\eta$--invariant, since they always admit orientation reversing isometries. This is enough to conclude that $\mathcal{M}_\lambda$ is indeed $\mathfrak{F}_2$: the only Euclidean $3$-manifold with $\beta^1(M_\lambda) = 1$ (see \cite{HW} or \cite[Corollary 3.5.10]{Wolf}) and vanishing $\eta$--invariant (see \cite{Ouyang} or \cite[\S 2.5]{Martelli}).
\end{proof}

\subsection{$DJ$-equivalence and isometry classes of the colourings of the cube}
We now refine the analysis of the previous section, and classify up to isometry all the orientable manifolds arising from proper colourings of the unit cube. 

The first step is to produce a classification up to $DJ$-equivalence.  
To do so, we generate all possible orientable colourings of the unit cube using a computer and find representatives for all $DJ$-equivalence classes. For the reader's convenience, a \texttt{SageMath} worksheet which performs this computation is available on \href{https://github.com/sashakolpakov/24-cell-colouring}{\texttt{GitHub}} \cite{github}. 

Notice that the dimension of the vector space is constrained between $3$ (the rank of a small cover) and $6$ (the total number of faces of the cube). We then analyse the topology of the resulting manifolds. Finally, we distinguish colourings of the same rank and homomorphism type using properties which can be directly read off of the defining matrix and that are clearly invariant under $DJ$-equivalence.

\bigskip

\begin{center}
\textbf{Rank 3}
\end{center}
\begin{table}[h]
\begin{tabular}{ll}
\noindent \textbf{A 3-torus:}
$\begin{pmatrix}
1 & 1 & 0 & 0 & 0 & 0 \\
0 & 0 & 1 & 1 & 0 & 0 \\
0 & 0 & 0 & 0 & 1 & 1
\end{pmatrix}$, & 
\noindent and \textbf{a half-twist manifold:} 
$\begin{pmatrix}
1 & 1 & 0 & 0 & 0 & 1 \\
0 & 0 & 1 & 1 & 0 & 1 \\
0 & 0 & 0 & 0 & 1 & 1
\end{pmatrix}$
\end{tabular}
\end{table}

Thus, there are only two types of manifolds arising from the cube as small covers. More equivalence classes appear in higher ranks.  

\bigskip
\begin{center}
\newpage
\textbf{Rank 4}
\end{center}

\bigskip
\noindent \textbf{3-tori:}
\begin{table}[h]
\begin{tabular}{lll}
(1) $\begin{pmatrix}
1 & 0 & 0 & 0 & 0 & 0 \\
0 & 0 & 1 & 1 & 0 & 0 \\
0 & 0 & 0 & 0 & 1 & 1 \\
0 & 1 & 0 & 0 & 0 & 0
\end{pmatrix}$, &

(2) $\begin{pmatrix}
1 & 0 & 0 & 0 & 0 & 1 \\
0 & 0 & 1 & 1 & 0 & 0 \\
0 & 0 & 0 & 0 & 1 & 1 \\
0 & 1 & 0 & 0 & 0 & 1
\end{pmatrix}$, &

(3) $\begin{pmatrix}
1 & 0 & 0 & 1 & 0 & 1 \\
0 & 0 & 1 & 1 & 0 & 0 \\
0 & 0 & 0 & 0 & 1 & 1 \\
0 & 1 & 0 & 1 & 0 & 1
\end{pmatrix}$
\end{tabular}
\end{table}

These three tori can be distinguished by looking at the %
number of pairs of opposite faces with linearly independent colours. In case (1) it is just one pair, two pairs in case (2) and three pairs in case (3).

\bigskip

\noindent \textbf{Half-twist manifolds:}
\begin{table}[h!]
\begin{tabular}{ll}
(1) $\begin{pmatrix}
1 & 0 & 0 & 0 & 0 & 0 \\
0 & 0 & 1 & 1 & 0 & 1 \\
0 & 0 & 0 & 0 & 1 & 1 \\
0 & 1 & 0 & 0 & 0 & 1
\end{pmatrix}$, &

(2) $\begin{pmatrix}
1 & 0 & 0 & 0 & 0 & 1 \\
0 & 0 & 1 & 1 & 0 & 0 \\
0 & 0 & 0 & 1 & 1 & 1 \\
0 & 1 & 0 & 1 & 0 & 1
\end{pmatrix}$
\end{tabular}
\end{table}

To distinguish (1) from (2) we use the same criterion that we used for the rank 4 tori: %
there are two linearly independent pairs in (1) while there are three in (2).

\bigskip
\noindent \textbf{A Hantzsche-Wendt manifold:}
Finally, there is a unique colouring of the cube which yields the Hantzsche-Wendt manifold given below.

\bigskip

\begin{table}[h!]
\begin{tabular}{l}
$\begin{pmatrix}
1& 0& 0& 0& 0& 1\\
0& 0& 1& 1& 0& 1\\
0& 0& 0& 1& 1& 0\\
0& 1& 0& 1& 0& 1
\end{pmatrix}$
\end{tabular}
\end{table}

\bigskip
\begin{center}
\textbf{Rank 5}
\end{center}

\bigskip
\noindent \textbf{3-tori:}

\begin{tabular}{lll}
(1) $\begin{pmatrix}
1 & 0 & 0 & 0 & 0 & 0 \\
0 & 0 & 1 & 0 & 0 & 0 \\
0 & 0 & 0 & 0 & 1 & 1 \\
0 & 1 & 0 & 0 & 0 & 0 \\
0 & 0 & 0 & 1 & 0 & 0
\end{pmatrix}$, &

(2) $\begin{pmatrix}
1 & 0 & 0 & 0 & 0 & 1 \\
0 & 0 & 1 & 0 & 0 & 0 \\
0 & 0 & 0 & 0 & 1 & 1 \\
0 & 1 & 0 & 0 & 0 & 1 \\
0 & 0 & 0 & 1 & 0 & 0
\end{pmatrix}$, &

(3) $\begin{pmatrix}
1 & 0 & 0 & 0 & 0 & 1 \\
0 & 0 & 1 & 0 & 0 & 1 \\
0 & 0 & 0 & 0 & 1 & 1 \\
0 & 1 & 0 & 0 & 0 & 1 \\
0 & 0 & 0 & 1 & 0 & 1
\end{pmatrix}$
\end{tabular}

\bigskip

In order to distinguish (1) from (2) and (3) we again look at the dimensions of the vector spaces generated by the colours assigned to the pairs of opposite faces of the cube. %
There are two pairs of linearly independent colours in (1) and three in (2) and (3).

\bigskip

Now, to tell (2) apart from (3), we notice that all rows in (3) have an even number of $1$ entries, while in (2) there are two rows with an odd number of $1$ entries. This means that if we multiply these two matrices by the vector $\varepsilon = \sum^6_{i=1} e_i$ on the right, we obtain the zero vector in  case (3) and a non-zero vector in case (2). This property is invariant under $DJ$--equivalence, since this amounts to changing the matrices by either multiplying by a permutation matrix on the right (which fixes the vector $\varepsilon$) or by multiplying by an invertible matrix on the left (which maps non-zero vectors to non-zero vectors).

\bigskip

\noindent \textbf{A half-twist manifold:} There is only one class up to $DJ$--equivalence. 

\begin{center}
$\begin{pmatrix}
1 & 0 & 0 & 0 & 0 & 1 \\
0 & 0 & 1 & 0 & 0 & 1 \\
0 & 0 & 0 & 0 & 1 & 1 \\
0 & 1 & 0 & 0 & 0 & 0 \\
0 & 0 & 0 & 1 & 0 & 0
\end{pmatrix}$
\end{center}

\bigskip
\begin{center}
\textbf{Rank 6}
\end{center}
\noindent Finally, there is a unique rank 6 colouring which produces a $3$-torus: it corresponds to the $6\times 6$ identity matrix.

\bigskip
Now we focus our attention on the classification of the colourings of the $3$-cube up to isometry. Equivalent colourings clearly produce isometric manifolds, so we ask ourselves if two different $DJ$-equivalence classes can be isometric. We have the following.

\begin{prop}
All $DJ$--equivalence classes of colourings of the $3$-cube are pairwise non-isometric.
\end{prop}

\begin{proof}
Non-homeomorphic manifolds are non-isometric. Manifolds with colourings of different rank are non-isometric either, since they are orbifold covers of the same orbifold with different degrees, and therefore have different volume. It is enough then to distinguish distinct $DJ$-equivalence classes with the same homeomorphism type and rank of the colouring. To do this, we first find fundamental domains and side-pairings for each manifold, using the methods in \cite[Section 3]{KS?}. We then look at the maximal translation sublattices in each lattice, and then check the length spectra of geodesics in the corresponding tori. 

The latter are generated as the lengths of integer linear combinations of the generating translations. Those spectra are pairwise distinct for each set of $DJ$-equivalence classes with the same homeomorphism type and rank, and therefore any two distinct $DJ$-equivalence classes of proper colourings of $\mathcal{C}$ correspond to a pair of non-isometric manifolds.
\end{proof}

\section{Constructing the 24-cell manifold}\label{sec:24-cell-coloring}

The regular Euclidean $24$-cell $\mathcal{Z}$ is a $4$-dimensional polytope which is realised as the convex hull in $\mathbb{R}^4$ of the following $24$ points:
\begin{equation*}
    (\pm1,0,0,0),\, (0,\pm1,0,0),\, (0,0,\pm1,0),\, (0,0,0,\pm1),\, 1/2\cdot(\pm1,\pm1,\pm1,\pm1).
\end{equation*}
The $24$-cell is self-dual, with $24$ octahedral facets and $96$ triangular faces. By interpreting the unit sphere in $\mathbb{R}^4$ as the boundary at infinity of hyperbolic space and taking the convex hull in $\mathbb{H}^4$ of these $24$ points we obtain the regular ideal hyperbolic $24$-cell, which is a regular right-angled hyperbolic polytope with $24$ ideal vertices, volume $\frac{4}{3} \pi^2$ and Euler characteristic $1$. Let this latter hyperbolic polytope be denoted by $\mathcal{Z}$. The context in which this notation is used should always resolve the ambiguity. For each vertex $p \in \mathcal{Z}$, its Euclidean vertex figure is a $3$-dimensional cube. 

Below, we exploit the great symmetry and associated algebraic structure of the regular $24$-cell $\mathcal{Z}$ in order to obtain a colouring that satisfies the following theorem. 

\begin{thm}\label{thm:main}
There exists a proper colouring $\lambda: \Gamma(\mathcal{Z}) \rightarrow \mathbb{F}^4_2$ of the ideal right-angled $24$-cell $\mathcal{Z} \subset \mathbb{H}^4$ such that the associated manifold $\mathcal{M} = \mathbb{H}^4 / \ker \lambda$ is orientable, and all of its cusp sections are homeomorphic to the Hantzsche-Wendt manifold. 
\end{thm}

\begin{proof} Our first step will be finding a suitable realisation of the ``combinatorial'' $24$-cell $\mathcal{Z}$ that allows for an elegant and succinct description of the necessary colouring. Since the $24$-cell is self-dual, in order to give a colouring to $\mathcal{Z}$ it is sufficient to assign colours in $\mathbb{F}^4_2$ to its vertices, and the properness conditions translates to the condition that the colours assigned to a triple of vertices that span a triangle are linearly independent.

Points of $\mathbb{R}^4$ can be realised as elements of Hamilton's quaternion algebra $\boldsymbol{H}$ by mapping the point $(x_1,x_2,x_3,x_4)$ to $x_1 \cdot 1 + x_2 \cdot \mathbf{i}\ + x_3 \cdot \mathbf{j} + x_4 \cdot \mathbf{k}$.
With this map, the vertices of $\mathcal{Z}$ are mapped to unit quaternions and inherit a group structure (the operation being defined by quaternion multiplication), which makes them isomorphic to the Hurwitz group $\mathfrak{H}$ consisting of the $24$ integral and half-integral quaternions of unit norm \cite[Tables 12.2 -- 12.3]{Martelli-book}.

The Hurwitz quaternion group has presentation \cite[Table 1, p.~134]{CoxeterMoser}
\begin{equation}\label{eq:hurwitz-presentation}
    \mathfrak{H} \cong \langle s,t \,| \, (st)^2 = s^3 = t^3 \rangle,
\end{equation}
and can be generated by taking $s = \frac{1}{2} (1+\mathbf{i}+\mathbf{j}+\mathbf{k})$, $t = \frac{1}{2} (1+\mathbf{i}+\mathbf{j}-\mathbf{k})$. This group is also isomorphic to $SL(2, \mathbb{F}_3)$, and is sometimes called the binary tetrahedral group, as it is isomorphic to the preimage in $\mathrm{Spin}(3)$ of the group of orientation preserving symmetries of a regular tetrahedron.

The group $\mathfrak{H}$ acts on $\mathbb{R}^4$ through left and right quaternion multiplication. Since all elements of $\mathfrak{H}$ are unit quaternions, this is an action by isometries. We can therefore define a left action of $\mathfrak{H} \times \mathfrak{H}$ on $\boldsymbol{H} \cong \mathbb{R}^4$ by setting
$(q_1,q_2)\cdot x = q_1 \cdot x \cdot q_2^{-1}$.

The action above is not faithful: its kernel is the cyclic subgroup of order $2$ generated by the element $(-1,-1)$. It preserves the vertices of $\mathcal{Z}$ and defines an index-two subgroup of the orientation preserving symmetry group of $\mathcal{Z}$ isomorphic to $\mathfrak{H} \times \mathfrak{H}/_{<(-1,-1)>}$. 

Now, we notice that there exists an injective homomorphism $\psi$ from the group $\mathfrak{H}$ into the group $\mathrm{GL}(4, \mathbb{F}_2)$ with image the group $H$ described in Table~\ref{tab:24-cell}. This homomorphism can be obtained, for example, via
\begin{equation*}
    s \mapsto \left(\begin{array}{rrrr}
0 & 0 & 1 & 1 \\
0 & 1 & 0 & 1 \\
1 & 1 & 0 & 1 \\
0 & 1 & 0 & 0
\end{array}\right), \,\,\,
t \mapsto \left(\begin{array}{rrrr}
1 & 1 & 0 & 0 \\
0 & 0 & 0 & 1 \\
1 & 0 & 0 & 0 \\
1 & 0 & 1 & 0
\end{array}\right).
\end{equation*}

Indeed, a simple computation is enough to check that $\psi$ respects the relations of \eqref{eq:hurwitz-presentation}.

The kernel of $\psi$ is among the proper normal subgroups of $\mathfrak{H}$, which amount to the trivial group, the centre $\{ \pm 1 \}$, and $Q_8 = \{ \pm 1, \pm \mathbf{i}, \pm \mathbf{j}, \pm \mathbf{k} \}$. 

A straightforward computation shows
\begin{equation*}
    \psi(-1) = M =
\begin{pmatrix}
0& 1& 1& 1 \\
0& 1& 0& 0 \\
1& 1& 0& 1\\
0& 0& 0& 1
\end{pmatrix},
\end{equation*}
which implies that $\psi$ is a monomorphism. 
A computational proof of the injectivity of $\psi$ by using \texttt{SageMath} \cite{SageMath} is available on \href{https://github.com/sashakolpakov/24-cell-colouring}{\texttt{GitHub}} \cite{github}.

We can therefore label each vertex $q \in \mathfrak{H}$ of the $24$-cell with the corresponding matrix $\phi(q) \in H$. In the quaternion notation, two elements $p, q \in \mathfrak{H}$ represent adjacent vertices if and only if the real part of $p q^{-1}$ equals $\frac{1}{2}$. This follows readily from regularity of the $24$-cell and the fact that the vertices of the $24$--cell that are adjacent to the vertex $1 \in \mathfrak{H}$ are exactly those with real part $\frac{1}{2}$. 

Note that an element of $\mathfrak{H}$ has real part equal to $\frac{1}{2}$ if and only if it has order $6$. Thus, $p$ and $q$ being adjacent is equivalent to the element $p q^{-1}$ having order $6$. 
By applying the homomorphism $\psi$, we obtain an alternative combinatorial model for the $24$-cell such that
\begin{enumerate}
\item[(i)] the vertices of $\mathcal{Z}$ are labeled by the $24$ matrices in $H < \mathrm{GL}(4, \mathbb{F}_2)$;
\item[(ii)] two matrices $M$ and $N$ correspond to adjacent vertices if and only if $M\cdot N^{-1}$ has order $6$ in $H$;
\item[(iii)] there is a left action of $H \times H$ on $\mathcal{Z}$ with kernel the cyclic group of order $2$ generated by 
$(\psi(-1),\psi(-1)) = (M,M) $.
\end{enumerate}

The $\mathbb{F}_2^4$-colouring on the vertices of $\mathcal{Z}$ is now given via the matrix model using the following very simple formula:
\begin{eqnarray}\label{eq:coloring}
\left.\begin{aligned}
\lambda(P) &= P\cdot e_1 , \text{ for each vertex }  P \in H < \mathrm{GL}(4, \mathbb{F}_2). 
\end{aligned}\right.
\end{eqnarray} where $e_1$ denotes the first vector of the standard basis of $\mathbb{F}_2^4$.

\begin{table}
\vspace{0.5in}
\centering
\begin{gather*}
v_1 = \left(\begin{array}{rrrr}
1 & 0 & 0 & 0 \\
0 & 1 & 0 & 0 \\
0 & 0 & 1 & 0 \\
0 & 0 & 0 & 1
\end{array}\right),\,\,
v_2 = \left(\begin{array}{rrrr}
1 & 0 & 0 & 1 \\
0 & 1 & 0 & 0 \\
1 & 1 & 0 & 0 \\
1 & 1 & 1 & 0
\end{array}\right),\,\,
v_3 = \left(\begin{array}{rrrr}
0 & 0 & 1 & 0 \\
1 & 0 & 1 & 0 \\
0 & 0 & 1 & 1 \\
0 & 1 & 0 & 0
\end{array}\right),\\
v_4 = \left(\begin{array}{rrrr}
0 & 1 & 1 & 1 \\
0 & 1 & 0 & 1 \\
0 & 1 & 1 & 0 \\
1 & 0 & 1 & 1
\end{array}\right),\,\,
v_5 = \left(\begin{array}{rrrr}
0 & 1 & 1 & 0 \\
0 & 0 & 0 & 1 \\
1 & 1 & 0 & 1 \\
0 & 1 & 0 & 1
\end{array}\right),\,\,
v_6 = \left(\begin{array}{rrrr}
1 & 1 & 0 & 1 \\
1 & 1 & 1 & 0 \\
1 & 0 & 0 & 1 \\
0 & 1 & 0 & 1
\end{array}\right),\\
\left(\begin{array}{rrrr}
0 & 1 & 1 & 1 \\
0 & 1 & 0 & 0 \\
1 & 1 & 0 & 1 \\
0 & 0 & 0 & 1
\end{array}\right),\,\,
\left(\begin{array}{rrrr}
0 & 0 & 1 & 1 \\
1 & 0 & 1 & 1 \\
1 & 0 & 0 & 1 \\
1 & 1 & 1 & 0
\end{array}\right),\,\,
\left(\begin{array}{rrrr}
0 & 1 & 1 & 0 \\
0 & 1 & 0 & 0 \\
0 & 0 & 1 & 1 \\
1 & 1 & 1 & 0
\end{array}\right),\,\,
\left(\begin{array}{rrrr}
0 & 0 & 1 & 0 \\
1 & 0 & 1 & 1 \\
0 & 1 & 1 & 1 \\
0 & 0 & 0 & 1
\end{array}\right),\\
\left(\begin{array}{rrrr}
1 & 1 & 0 & 0 \\
1 & 0 & 1 & 1 \\
0 & 1 & 1 & 0 \\
1 & 1 & 1 & 0
\end{array}\right),\,\,
\left(\begin{array}{rrrr}
1 & 1 & 0 & 1 \\
1 & 0 & 1 & 1 \\
1 & 0 & 0 & 0 \\
0 & 0 & 0 & 1
\end{array}\right),\,\,
\left(\begin{array}{rrrr}
1 & 0 & 0 & 1 \\
1 & 0 & 1 & 0 \\
0 & 1 & 1 & 1 \\
1 & 0 & 1 & 1
\end{array}\right),\,\,
\left(\begin{array}{rrrr}
1 & 1 & 0 & 0 \\
0 & 0 & 0 & 1 \\
1 & 0 & 0 & 0 \\
1 & 0 & 1 & 0
\end{array}\right),\\
\left(\begin{array}{rrrr}
0 & 0 & 1 & 1 \\
0 & 1 & 0 & 1 \\
1 & 1 & 0 & 1 \\
0 & 1 & 0 & 0
\end{array}\right),\,\,
\left(\begin{array}{rrrr}
0 & 1 & 1 & 1 \\
1 & 1 & 1 & 0 \\
1 & 1 & 0 & 0 \\
1 & 0 & 1 & 0
\end{array}\right),\,\,
\left(\begin{array}{rrrr}
1 & 1 & 0 & 0 \\
0 & 1 & 0 & 1 \\
0 & 0 & 1 & 0 \\
0 & 1 & 0 & 0
\end{array}\right),\,\,
\left(\begin{array}{rrrr}
1 & 0 & 0 & 0 \\
1 & 1 & 1 & 0 \\
0 & 0 & 1 & 1 \\
1 & 0 & 1 & 0
\end{array}\right),\\
\left(\begin{array}{rrrr}
1 & 1 & 0 & 1 \\
1 & 0 & 1 & 0 \\
1 & 1 & 0 & 0 \\
0 & 1 & 0 & 0
\end{array}\right),\,\,
\left(\begin{array}{rrrr}
0 & 0 & 1 & 0 \\
1 & 1 & 1 & 0 \\
0 & 1 & 1 & 0 \\
0 & 1 & 0 & 1
\end{array}\right),\,\,
\left(\begin{array}{rrrr}
1 & 0 & 0 & 0 \\
0 & 1 & 0 & 1 \\
1 & 0 & 0 & 1 \\
1 & 0 & 1 & 1
\end{array}\right),\,\,
\left(\begin{array}{rrrr}
1 & 0 & 0 & 1 \\
0 & 0 & 0 & 1 \\
0 & 0 & 1 & 0 \\
0 & 1 & 0 & 1
\end{array}\right),\\
\left(\begin{array}{rrrr}
0 & 1 & 1 & 0 \\
1 & 0 & 1 & 0 \\
1 & 0 & 0 & 0 \\
1 & 0 & 1 & 1
\end{array}\right),\,\,
\left(\begin{array}{rrrr}
0 & 0 & 1 & 1 \\
0 & 0 & 0 & 1 \\
0 & 1 & 1 & 1 \\
1 & 0 & 1 & 0
\end{array}\right).
\end{gather*}
\caption{The matrices in $\mathrm{GL}(4, \mathbb{F}_2)$ as vertices of the regular $24$--cell}\label{tab:24-cell}
\end{table}

Before checking for properness of the colouring and looking at the cusp sections we notice that, by construction, the group of admissible symmetries $\mathrm{Adm}_{\lambda}(\mathcal{Z})$ contains $H$ as a subgroup acting on $\mathcal{Z}$ via \textit{left} multiplication. Indeed, left multiplication by a matrix $M \in H$ permutes the colours associated to the vertices precisely through multiplication by $M$, which is obviously a linear map.  

A more detailed check shows that $\mathrm{Adm}_{\lambda}(\mathcal{Z}) = H \times C_3$. Here $C_3$ is the cyclic group of order $3$ that acts on the vertices of $\mathcal{Z}$ via \textit{right} multiplication by 
\begin{equation*}
S = \left(\begin{array}{rrrr}
1 & 1 & 0 & 0 \\
0 & 1 & 0 & 1 \\
0 & 0 & 1 & 0 \\
0 & 1 & 0 & 0
\end{array}\right).
\end{equation*}

Indeed, notice that $S \in \psi(\mathfrak{H})$ according to Table~\ref{tab:24-cell}. Since the first column of $S$ is the vector $e_1$, then right multiplication by $S$ preserves the colours of the facets. A complete search for admissible symmetries following Definition \ref{def:admissible-symmetries} then proves that there are no other but the ones indicated above. The reader can find this computation performed with \texttt{SageMath} \cite{SageMath} available on \href{https://github.com/sashakolpakov/24-cell-colouring}{\texttt{GitHub}} \cite{github}. 

Therefore the induced homomorphism 
\begin{equation}\label{eq:homomorphism}
\Phi: \mathrm{Adm}_{\lambda}(\mathcal{Z}) = H \times C_3\rightarrow \mathrm{GL}(4,\mathbb{F}_2)\end{equation} 
is the identity on the left factor, and is trivial on the right factor.

We now have to check that the colouring of $\mathcal{Z}$ is proper. For this purpose, let us notice that $O = \{ v_1, v_2, v_3, v_4, v_5, v_6 \}$ (with the $v_i$'s as in Table~\ref{tab:24-cell}) forms an octahedron in $\mathcal{Z}$. This octahedron is dual to an (ideal) vertex $p$ of the ideal right-angled $24$-cell $\mathcal{Z}$ and therefore its vertices correspond to the generators of the maximal affine Coxeter subgroup $\Gamma_p<\Gamma(\mathcal{Z})$. By the discussion in Section \ref{sec:ideal-vertices}, the induced colouring on $\Gamma_p$ will determine the isometry class of a Euclidean cusp section of the orbifold $\mathbb{H}^4/\ker \lambda$. This latter colouring is easily seen to be proper, with the defining matrix  
\begin{equation*}
\Lambda = \left(\begin{array}{cccccc}
1& 1& 0& 0& 0& 1\\
0& 0& 1& 0& 0& 1\\
0& 1& 0& 0& 1& 1\\
0& 1& 0& 1& 0& 0
\end{array}\right).
\end{equation*}
The space $\mathrm{Row}(\Lambda)$ satisfies the following equations:
\begin{equation*}
    x_1+x_3+x_5+x_6=0, \, x_1+x_2+x_4+x_5=0
\end{equation*} and it is easy to check that no vector from the set $T=\{e_1+e_2,e_3+e_4,e_5+e_6\}$ belongs to $\mathrm{Row(\Lambda)}$. By Proposition \ref{prop:cube} the corresponding $DJ$-equivalence class is that of the rank $4$ colouring producing the Hantzsche-Wendt manifold.

Then the rest of the colouring is proper, since each octahedron in $\mathcal{Z}$ can be obtained from $O$ by left multiplication by $H$ and  each triangle in $\mathcal{Z}$ is an image of a triangle in $O$, thus $\mathcal{M}=\mathbb{H}^4 / \ker \lambda$ is a manifold. By computing the dimensions of the colourings of $\mathcal{Z}$ and $O$ (both equal to $4$), we conclude that $\mathcal{M}$ has $24$ cusps in total \cite[Proposition~2.2]{KS?}, all of whose sections are homeomorphic to the Hantzsche-Wendt manifold. Indeed, the coloured isometry group of $\mathcal{M}$ acts transitively on the cusps, again because the group of admissible symmetries acts transitively on the octahedra.

It is also easy to verify that $\lambda$ is an orientable colouring: since all its colours have an odd number of $1$ entries, the vector $\varepsilon=(1,\dots,1)$ belongs to $\mathrm{Row}(\lambda)$ and Proposition \ref{prop:orientability_criterion} applies. Thus $\mathcal{M}$ is an orientable manifold.
\end{proof}

\subsection{The geometry and topology of $\mathcal{M}$}
We conclude this paper by providing the reader with some more information on the manifold $\mathcal{M}$. The orbifold cover $\mathcal{M}\rightarrow \mathbb{H}^4/\Gamma(\mathcal{Z})$ corresponds to the kernel of a homomorphism from $\Gamma(\mathcal{Z})$ onto $\mathbb{F}_2^4$. Thus $\mathcal{M}$ is tessellated by $16$ copies of $\mathcal{Z}$: its volume equals $16 \cdot \frac{4}{3} \pi^2 = \frac{64}{3} \pi^2$ and its Euler characteristic is $\chi(\mathcal{M})=16$.

The tessellation of $\mathcal{M}$ into copies of $\mathcal{Z}$ is an Epstein-Penner decomposition for $\mathcal{M}$ with respect to the maximal cusp section (see \cite[Proposition 2.5]{KM} for details), and thus is preserved by $\mathrm{Isom}(\mathcal{M})$. Any isometry of $\mathcal{M}$ which preserves the tessellation into copies of $\mathcal{Z}$ lies in the coloured isometry group of $\mathcal{M}$, and therefore we are left with the task of computing this latter group. We have all the ingredients for this:
by (\ref{eq:semidirect_product}) and (\ref{eq:homomorphism}) the coloured isometry group of $\mathcal{M}$ is isomorphic to
\begin{equation}\label{eq:isometry group}
(\mathbb{F}_2^4 \rtimes H) \times C_3 \cong \mathrm{Isom}(\mathcal{M})
\end{equation} with the action of $H$ on $\mathbb{F}_2^4$ given by matrix multiplication.

By applying Corollary \ref{cor:CP} we compute the Betti numbers of $\mathcal{M}$. Namely,
\begin{equation*}
    \beta^0=1,\,\beta^1=0, \,\beta^2=38, \beta^3=23, \beta^4=0,
    \end{equation*}
which is consistent with the fact that $\mathcal{M}$ has Euler characteristic $16$.    

Finally, we remark the fact that the colouring described in equation (\ref{eq:coloring}) is the \emph{only} orientable rank $4$ colouring (up to $DJ$--equivalence) of the ideal right-angled $24$-cell $\mathcal{Z}$ which produces a manifold with all cusp sections homeomorphic to the Hantzsche-Wendt manifold. We also found a single rank $5$ $DJ$-equivalence class of colourings of the $24$-cell with these properties, but no such rank $6$ colourings. These facts were checked by the authors computationally using \texttt{SageMath} \cite{SageMath}.

\end{document}